\documentclass[reqno,11pt]{amsproc}

\usepackage{graphicx}
\usepackage{tikz}
\usepackage{tikz-cd}
\usepackage[margin=1.5in]{geometry}
\usepackage{amssymb}
\usepackage{hyperref}
\usepackage{microtype}

\usepackage{enumitem}
\usepackage[nocompress]{cite}
\usepackage{bbm}

\usepackage[skip=1ex]{caption}
\usepackage{subcaption}

\newtheorem{theorem}{\bf Theorem}[section]
\newtheorem{proposition}[theorem]{\bf Proposition}
\newtheorem{lemma}[theorem]{\bf Lemma}
\newtheorem{corollary}[theorem]{\bf Corollary}

\theoremstyle{definition}
\newtheorem{definition}[theorem]{\bf Definition}
\newtheorem*{definition*}{\bf Definition}
\newtheorem{remark}[theorem]{\bf Remark}
\newtheorem{example}[theorem]{\bf Example}
\newtheorem{conjecture}[theorem]{\bf Conjecture}

\renewcommand{\descriptionlabel}[1]%
     {\hspace{\labelsep}\textsf{#1}}

\newcommand		{\p}[1]	{\left(#1\right)}

\newcommand		{\abs}[1]{\left|#1\right|}

\newcommand{\T} {\mathbb T}
\newcommand{\R} {\mathbb R}
\newcommand{\C} {\mathbb C}
\newcommand{\Z} {\mathbb Z}

\newcommand{\G} {\mathbf G}
\renewcommand{\H} {\mathbf H}
\newcommand{\N} {\mathbf N}
\newcommand{\V} {\mathbf V}
\newcommand{\E} {\mathbf E}
\newcommand{\J} {\mathbf J}
\newcommand{\K} {\mathbf K}

\DeclareTextFontCommand{\emph}{\bfseries\em}

\begin{document}

\setcounter{MaxMatrixCols}{12}

\title[Kuramoto Networks with Infinitely Many Stable Equilibria]
{Kuramoto Networks with \\Infinitely Many Stable Equilibria}

\author[Davide Sclosa]{Davide Sclosa$^1$}

\footnotetext[1]{Mathematics Department, Vrije Universiteit Amsterdam;
e-mail:\hfill{\mbox{}}\\\hbox{d.sclosa@vu.nl}}


\begin{abstract}
We prove that the Kuramoto model on a graph can
contain infinitely many non-equivalent stable equilibria.
More precisely, we prove that for every~$d\geq 1$ there is a connected graph
such that the set of stable equilibria contains a manifold of dimension~$d$.
In particular, we solve a conjecture of R.~Delabays, T.~Coletta and P.~Jacquod
about the number of equilibria on planar graphs.
Our results are based on the analysis of balanced configurations, which
correspond to equilateral polygon linkages in topology.
In order to analyze the stability of manifolds of equilibria
we apply topological bifurcation theory.
\end{abstract}

\maketitle

\section{Introduction.}
Consider a connected graph~$\G$ with vertices~$1,\ldots,n$ and
to each vertex~$j$ associate a \emph{phase}~$\theta_j$
in the $1$-dimensional torus~$\T=\R/2\pi\Z$.
Let~$\mathbf N(j)$ denote the set of neighbors of~$j$ and consider
the coupled dynamical system
\begin{equation} \label{eq:main}
	\dot \theta_j = \sum_{k\in \mathbf N(j)} \sin(\theta_k - \theta_j),
	\qquad \text{for all } j=1,\ldots,n.
\end{equation}

In the paper \emph{stable} always means Lyapunov stable.
For every graph~$\G$ the synchronized state, in which all the phases are equal,
is a stable equilibrium. It is known that other stable equilibria are present in
cycles~\cite{canale2009, Wiley2006},
planar graphs~\cite{delabays2017multistability},
sparse graphs~\cite{sokolov2019sync},
$3$-regular graphs~\cite{DeVille2016}.
Two equilibria are \emph{equivalent} if they differ by a constant.
The number of non-equivalent stable equilibria is typically understood to be finite,
and explicit bounds are known in some cases~\cite{delabays2017multistability}.

On the other hand, it is known that some graphs support infinitely many
non-equivalent unstable equilibria.
Indeed, unstable equilibria form a manifold with singularities in the case of
complete graphs~\cite{brown2003globally, watanabe1997stability, ashwin2016identical}.

This leads to the question motivating the paper:
is the number of non-equivalent stable equilibria on a connected graph 
always finite? Can stable equilibria form a manifold?
Our main result answers these question:

\begin{theorem} \label{thm:main}
For every~$d\geq 1$ there is a connected
graph~$\G$ such that~\eqref{eq:main} contains a manifold of stable equilibria
of dimension~$d$.
\end{theorem}

In~\cite{delabays2017multistability} R.~Delabays, T.~Coletta and P.~Jacquod
conjecture an upper bound on the number of non-equivalent stable equilibria
in connected, planar graphs. In this paper we show that no such bound is possible,
and in fact a connected,
planar graph can support infinitely many non-equivalent stable equilibria:

\begin{corollary}
The graph of Figure~\ref{fig:main} supports infinitely many non-equivalent
stable equilibria.
\end{corollary}

\begin{figure}
\centering
\includegraphics[scale=0.14, angle=0]{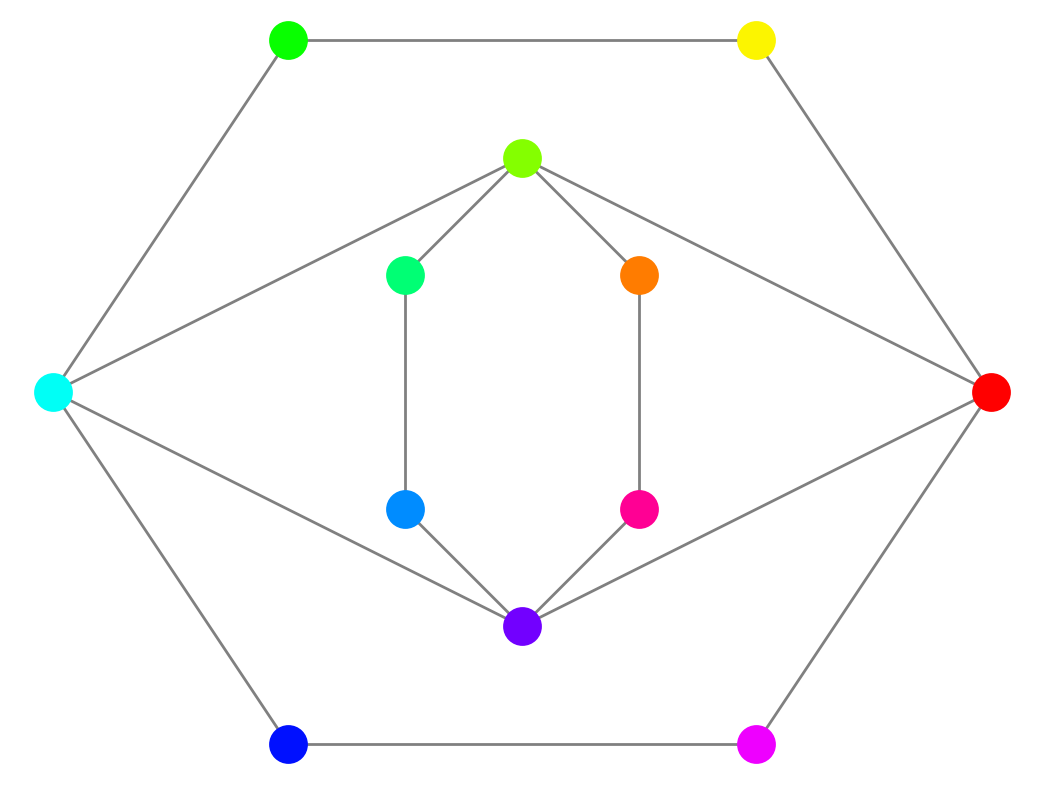}\label{fig:zero}
\caption{The \emph{eye graph}
supports infinitely many non-equivalent stable equilibria.
A cyclic colormap
is used to represent phases on vertices.
The phases of the outer cycle are~$(2k\pi/6)_{k=1,\ldots,6}$,
starting from the rightmost vertex and proceeding counterclockwise.
The phases of the inner cycle are~$(2k\pi/6 + \beta)_{k=1,\ldots,6}$ with~$\beta=\pi/2$,
starting from the topmost vertex and proceeding counterclockwise.
In Section~\ref{sec:balanced} we prove that
a curve of non-equivalent stable equilibria is obtained by varying~$\beta$ in a
neighborhood of~$\pi/2$.
}
\label{fig:main}
\end{figure}

In Section~\ref{sec:tools} we recall a known technique: algebraic geometrization.
Up to a change of coordinates, the set of equilibria of~\eqref{eq:main} turns into an algebraic set.
Every algebraic set is the finite union of irreducible algebraic sets; in particular, this
implies that infinitely many equilibria can only appear inside a continuum of equilibria.
Moreover, the system has a gradient structure:
every solution is either an equilibrium or a curve joining two
algebraic sets of equilibria, in the direction that makes some energy function decrease.
We return to this topological/heteroclinic structure in Section~\ref{sec:heteroclinic},
in which we analyze explicitly two examples of small cardinality.

Section~\ref{sec:balanced} concerns~\emph{balanced configurations},
the main topic of the paper.
We say that a subset of vertices~$\K$ is \emph{balanced} if~$\sum_{k\in \K} e^{i\theta_k} = 0$.
As we will see, balanced configurations are able to ``effectively disconnects'' the network,
leading to manifolds of equilibria of arbitrarily large dimension on connected graphs.
The main challenge we will face is designing graphs for which
these manifolds are transversally stable.

It is interesting to notice that balanced configurations appear in several areas of mathematics,
although with different names.
First, the quantity~$\sum_{k=1}^n e^{i\theta_k}$, known as~\emph{order parameter},
is widely used as a measure of
synchronization in phase oscillator networks~\cite{mirollo2012asymptotic, pazo2005thermodynamic, bick2011chaos, kuramoto1987statistical}.
Second, balanced configurations of are known as balanced graph representations
in algebraic graph theory~\cite{godsil2001algebraic}.
Third, balanced configurations are in $1$-to-$1$ correspondence
with equilateral polygon linkages in
topology~\cite{kamiyama1996topology, kamiyama1992elementary, kapovich1996symplectic, mandini2014duistermaat}.

In Section~\ref{sec:aligned} we discuss
\emph{aligned} configurations, those in which any two phases differ by~$0$ or~$\pi$.
They are appear in literature with different names~\cite{park2016weakly, martens2016chimera, kopell1995anti, vathakkattil2020limits, markdahl2021counterexamples, hendrickx2012convergence, ren2005consensus}.
In this paper we are mainly interested in the interplay between aligned configurations
and balanced configurations.
We prove that every equilibrium of a complete bipartite graph is a combination of these.

Our insights raise a number of questions.
We know that the set of equilibria can be written as a finite union of algebraic varieties.
Algebraic varieties are not necessarily manifolds, due to the presence of singular points.
However:

\begin{conjecture}
For every graph the set of equilibria of~\eqref{eq:main}
is a finite union of manifolds.
\end{conjecture}

Our analysis may extend to
Kuramoto networks on weighted graphs, hypergraphs, or with intrinsic frequencies.
A connection between the analysis of
balanced configurations on weighted graph and the theory of moduli spaces
in topology~\cite{kapovich1995moduli, mandini2014duistermaat, shimamoto2005spaces}
is outlined in Section~\ref{sec:balanced}.

Finally, manifolds of stable equilibria may appear on graphs
that are close to the global synchronization
constant~\cite{Ling2019, lu2020, Taylor2012, Yoneda2021}, thus
limiting the effectiveness of linear stability analysis as predicted in~\cite{Kassabov2021}.

\section*{Acknowledgements.}
The author would like to thank Christian Bick for many helpful discussions.



%


\section{A System Rich in Structure.} \label{sec:tools}
The combinatorial structure of the underlying graph,
together with the algebraic properties of the sine function,
give the coupled dynamical system~\eqref{eq:main} some known,
peculiar properties, which we review in detail in this section.

\subsection{Phase-Shift Symmetry and Connectivity.}
The equations~\eqref{eq:main} remain invariant if the same constant
is added to each phase. This (dynamical) symmetry is known as~\emph{phase shift} and
defines an action of the group~$\T$ on the phase space~$\T^n$.
The phase space~$\T^n$ is foliated into~$(n-1)$-dimensional dynamically
invariant tori. Each leaf supports the same dynamics and
the group~$\T$ acts freely on the leaves.

As a consequence, equilibria are never isolated:
every equilibrium is contained in a $1$-dimensional
torus of equilibria, its orbit under the group action.
We say that two equilibria are \emph{equivalent} if they belong to the same orbit.

If a graph is disconnected, distinct components have independent dynamics.
If every connected component is endowed with a stable equilibrium,
infinitely many non-equivalent stable equilibria
can be obtained by phase-shifting the phases of one component. In this way we can obtain
a torus of stable equilibria of dimension the number of connected components.
This is a non-interesting solution to the question motivating the paper.
For this reason, we will always require connectivity. 

\subsection{Gradient Descent.}
It is well known that~\eqref{eq:main} is a gradient
dynamical system~\cite{van1993lyapunov, Ling2019, jadbabaie2004stability}.
We can write~\eqref{eq:main} as~$\dot \theta = -\nabla E(\theta)$ where~$E:\T^N\to \R$
is the \emph{energy function}
\begin{equation} \label{eq:energy}
	E(\theta) = \sum_{jk \in \E(\G)} \p{1-\cos(\theta_j-\theta_k)}.
\end{equation}
Here~$\E(\G)$ denotes the set of edges and each edge is counted exactly once.

The identity~$\dot \theta = -\nabla E(\theta)$ tells that
a solution always evolves in the direction where the energy~\eqref{eq:energy}
decreases maximally. Equilibria are exactly the critical points of the energy
and the only periodic trajectories.
A solution is either an equilibrium
or a curve joining two equilibria, traveled in the direction in which the energy decreases.

The energy function~\eqref{eq:energy} is real analytic.
In any real analytic gradient system the Lyapunov stable equilibria are exactly
the local minima of the energy function~\cite{Absil2006}.
It is interesting to notice that in a smooth (but not real analytic)
gradient system this statement can fail in both ways~\cite{Absil2006}.

\subsection{Algebraic Geometry.}
The set of equilibria of~\eqref{eq:main} is best understood in the language of algebraic
geometry. For basic definitions and results we refer
the reader to~\cite{hartshorne2013algebraic, lang2019introduction, shafarevich1994basic}.
Let us identify the phase space~$\T^n$ with the subset of~$\R^{2n}$ defined by
\begin{equation} \label{eq:ag_torus}
	x_j^2 + y_j^2 = 1, \qquad \text{for all } j=1,\ldots,n
\end{equation}
where~$x_k = \cos(\theta_k)$ and~$y_k = \sin(\theta_k)$.
Then the equilibria of the system are
the common solutions of~\eqref{eq:ag_torus} and
\begin{equation} \label{eq:ag_equilibria}
	\sum_{k\in \N(j)} x_k y_j - x_j y_k = 0, \qquad \text{for all } j=1,\ldots,n.
\end{equation}
Therefore, the set of equilibria~$X$ is an algebraic set.
As such, it has a unique decomposition into irreducible components:
\begin{equation} \label{eq:ag_decomp}
	X = X_1 \cup \ldots \cup X_ m.
\end{equation}
Each~$X_i$ is an irreducible algebraic set (also known as algebraic variety)
and none of the~$X_i$ is contained in the union of
the others. Each~$X_i$ has a well defined dimension and tangent space at every point,
except for singular points, if any.

The $1$-dimensional components in~\eqref{eq:ag_decomp} are orbits of equivalent equilibria,
topologically they are $1$-dimensional tori.
Since the decomposition~\eqref{eq:ag_decomp} is finite, we immediately learn
something important about the equilibria.
Either there are finitely many non-equivalent equilibria,
or there is a component of dimension greater than~$1$.
In particular, the number of equilibria that are isolated up to phase-shift is finite.

Algebraic geometry does not prevent the existence of infinitely many non-equivalent equilibria.
By contrast, it helps us understanding their geometry.
The methods used in~\cite{baillieul1982geometric, Chen2018}
to prove the finiteness of the equilibria
apply to a (related but) different system and only work for a generic choice of coupling
coefficients.

\subsection{Stability.}
Every solution of~\eqref{eq:main} is either an equilibrium
or a curve joining two irreducible components.
A complete understanding of the system~\eqref{eq:main} requires understanding
the components~\eqref{eq:ag_decomp} and how they are connected by
solutions.

In order to do so, we need to analyze the stability of each component.
In this section we briefly review the analysis of $1$-components,
which is the case already known in literature.
Components of larger dimension will be analyzed in Section~\ref{sec:balanced}.

Since there are no isolated equilibria,
throughout the paper \emph{stable} will always mean Lyapunov stable.

Let~$(a_{jk})_{j,k}$ denote the adjacency matrix of the graph.
Fix any~$\theta = (\theta_j)_{j=1,\ldots,n}$.
The component~$(j,k)$ of the Jacobian matrix of the system at~$\theta$ is
\begin{equation} \label{eq:jacobian}
\begin{cases}
a_{jk} \cos(\theta_k-\theta_j) & \text{if } j\neq k \\
-\sum_{h,\, h\neq j} a_{jk} \cos(\theta_h-\theta_j) & \text{if } j=k.
\end{cases}
\end{equation}
For every~$\theta$ at least one Jacobian eigenvalue is zero;
it corresponds to the direction of phase-shift~$(1,\ldots,1)$.

Suppose now that~$\theta$ is an equilibrium, and that
exactly one eigenvalue is~$0$ and the others are strictly negative.
The zero eigenvalue corresponds to the direction tangent to the
$1$-dimensional torus of equilibria
\[
	\Gamma = \{(\theta_j + \alpha)_{j=1,\ldots,n} \mid \alpha\in\T \}.
\]
The zero eigenvalue disappears by restricting dynamics to the leaf containing~$\theta$,
in which the equilibrium is isolated.
In particular~$\theta$ is asymptotically stable in its leaf.

By phase shift symmetry, each equilibrium in~$\Gamma$ is asymptotically stable in its leaf.
The leaf locally coincides with the stable manifold of the equilibrium.
These manifolds foliates the space around~$\Gamma$
and intersect~$\Gamma$ orthogonally (since the Jacobian~\eqref{eq:jacobian} is symmetric).

\section{Balanced Configurations and Manifolds of Equilibria.} \label{sec:balanced}
This is the main section of the paper. For us, balanced configurations are a tool
to understand manifolds of equilibria. For their role in other branches of mathematics and science
we refer the reader to the introduction.

It is convenient for us to expand the notions of equilibrium and balanced configuration
to a proper subset of vertices.

\begin{definition} \label{def:equilibrium_balanced}
A \emph{configuration}
is a pair consisting of a set of vertices~$\K\subseteq \V(\G)$
and a vector of phases~$\theta_\K = (\theta_k)_{k\in \K}$.
An~\emph{equilibrium} of~$\K$ is a configuration satisfying
\begin{equation} \label{eq:equilibrium}
	\sum_{k\in \N(j)\cap \K} \sin(\theta_k - \theta_j) = 0, \qquad \text{for all } j\in \K.
\end{equation}
A configuration of~$\K$ is \emph{balanced} if
\begin{equation} \label{eq:balanced}
	\sum_{k\in \K} e^{i\theta_k} = 0.
\end{equation}
\end{definition}

The following lemma collects some elementary properties. The proof is omitted. 

\begin{lemma} \label{lem:basic} Let~$(\K, \theta_\K)$ be a configuration.
The following facts are true:
\begin{enumerate} [label = (\roman*)]
\item The equilibria of~$\V(\G)$ according to Definition~\ref{def:equilibrium_balanced}
are exactly the equilibria of the dynamical system~\eqref{eq:main};
\item For every~$\alpha\in\T$ the configuration~$\theta_\K$ is an equilibrium if and only
if the configuration~$\theta_\K+\alpha = (\theta_k+\alpha)_{k\in \K}$
is an equilibrium; \label{lem:basic:2}
\item For every~$\alpha\in\T$ the configuration~$\theta_\K$ is balanced if and only
if the configuration~$\theta_\K+\alpha = (\theta_k+\alpha)_{k\in \K}$
is balanced; \label{lem:basic:3}
\item Let~$j$ be any vertex and suppose that~$(\K, \theta_\K)$ is balanced. Then
\begin{align} 
	& \sum_{k\in \K} \cos(\theta_k - \theta_j) = 0, \label{eq:cosine} \\
	& \sum_{k\in \K} \sin(\theta_k - \theta_j) = 0. \label{eq:sine}
\end{align}
\end{enumerate}
\end{lemma}
\begin{proof}
Parts~$(i)$, $(ii)$, and~$(iii)$ are trivial.
In order to obtain~$(iv)$, multiply~\eqref{eq:balanced} by~$e^{-i\theta_j}$
and take real and imaginary part. This gives~\eqref{eq:cosine} and~\eqref{eq:sine}
respectively.
\end{proof}

Equation~\eqref{eq:sine} tells us that
the net input received by~$j$ from a balanced set of neighbors is zero.
This suggests how to obtain manifolds of equilibria of arbitrarily large dimension.

\begin{lemma} \label{lem:balanced}
Suppose that there is a partition of the vertices of~$\G$
into non-empty parts $\J_1,\ldots,\J_d$
and a configuration~$\theta=(\theta_j)_{j=1,\ldots,n}$ such that
\begin{enumerate} [label = (\roman*)]
\item every~$(\J_p, \theta_{\J_p})$ is an equilibrium; \label{lem:balanced:1}
\item for every~$\J_p$,~$\J_q$ distinct and for every~$j\in \J_p$
the set of neighbors of~$j$ in~$\J_q$ is balanced. \label{lem:balanced:2}
\end{enumerate}
Then for every~$\alpha_1,\ldots,\alpha_d\in \T$ the configuration
\begin{equation} \label{eq:d-shift}
	(\theta_{\J_1} + \alpha_1, \ldots, \theta_{\J_d} + \alpha_1)
\end{equation}
is an equilibrium of~$\G$. In particular, the set of equilibria of~$\G$
contains a torus of dimension~$d$.
\end{lemma}
\begin{proof}
Take any vertex~$j$. Suppose that~$j\in \J_p$. We have
\[
	\dot\theta_j = \sum_{k\in \N(j) \cap \J_p} \sin(\theta_j-\theta_k)
		+ \sum_{q\neq p} \sum_{k\in \N(j) \cap  \J_q} \sin(\theta_j-\theta_k).
\]
Since~$\theta_{\J_p}$ is an equilibrium of~$\J_p$, the first term is zero.
Since~$\N(j) \cap  \J_q$ is balanced for every~$q\neq p$, by~\eqref{eq:sine}
the second term is zero.
This shows that~$\theta$ is an equilibrium of~$\G$.
The statement now follows from
Lemma~\ref{lem:basic}\ref{lem:basic:2} and Lemma~\ref{lem:basic}\ref{lem:basic:3}.
\end{proof}

Let~$\Gamma \subseteq \T^n$ be the $d$-dimensional torus of equilibria~\eqref{eq:d-shift}.
In general, not all the equilibria in~$\Gamma$
have the same stability. As we will see, by
varying~$\alpha_1,\ldots,\alpha_d$ some Jacobian eigenvalues can change sign.
This phenomenon is a particular case of topological bifurcation~\cite{liebscher2015bifurcation}.
In the following proof we construct graphs
for which a non-empty open subset of~$\Gamma$ is transversally stable.

\begin{figure}
\centering
\begin{tikzpicture} [scale=0.8, rotate=0]
\draw [shift={(0,0)}]
	(1,0) -- (0.5,0.866) -- (-0.5,0.866) -- (-1,0) -- (-0.5,-0.866) -- (0.5,-0.866) -- (1,0);
\draw [shift={(4,4)}]
	(1,0) -- (0.5,0.866) -- (-0.5,0.866) -- (-1,0) -- (-0.5,-0.866) -- (0.5,-0.866) -- (1,0);
\draw [shift={(-4,4)}]
	(1,0) -- (0.5,0.866) -- (-0.5,0.866) -- (-1,0) -- (-0.5,-0.866) -- (0.5,-0.866) -- (1,0);
\draw [shift={(4,-4)}]
	(1,0) -- (0.5,0.866) -- (-0.5,0.866) -- (-1,0) -- (-0.5,-0.866) -- (0.5,-0.866) -- (1,0);
\draw [shift={(-4,-4)}]
	(1,0) -- (0.5,0.866) -- (-0.5,0.866) -- (-1,0) -- (-0.5,-0.866) -- (0.5,-0.866) -- (1,0);
\draw (1,0) to [out=15, in=-70, distance=60pt] (5,4);
\draw (1,0) to [out=30, in=-100, distance=60pt] (3,4);
\draw (-1,0) to [out=180-60, in=-90+10, distance=80pt] (5,4);
\draw (-1,0) to [out=180-45, in=180, distance=80pt] (3,4);
\draw (1,0) to [out=60, in=-100, distance=80pt] (-5,4);
\draw (1,0) to [out=45, in=0, distance=80pt] (-3,4);
\draw (-1,0) to [out=180-15, in=-110, distance=60pt] (-5,4);
\draw (-1,0) to [out=180-30, in=-90+10, distance=60pt] (-3,4);
\draw (1,0) to [out=-15, in=70, distance=60pt] (5,-4);
\draw (1,0) to [out=-30, in=100, distance=60pt] (3,-4);
\draw (-1,0) to [out=180+60, in=+90-10, distance=80pt] (5,-4);
\draw (-1,0) to [out=180+45, in=180, distance=80pt] (3,-4);
\draw (1,0) to [out=-60, in=100, distance=80pt] (-5,-4);
\draw (1,0) to [out=-45, in=0, distance=80pt] (-3,-4);
\draw (-1,0) to [out=180+15, in=+110, distance=60pt] (-5,-4);
\draw (-1,0) to [out=180+30, in=90-10, distance=60pt] (-3,-4);
\end{tikzpicture}
\caption{This figure illustrates the graph~$\G_5$.}
\label{fig:cycles}
\end{figure}
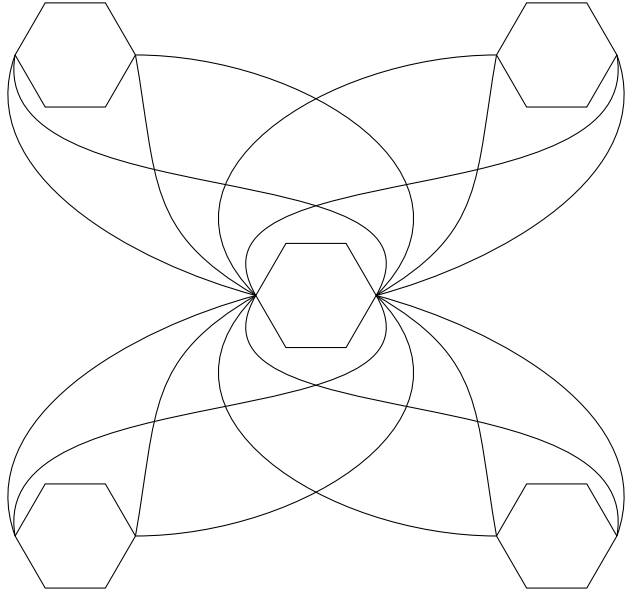

\begin{proof} [Proof of Theorem~\ref{thm:main}]
The idea is connecting a family of stable $6$-cycle configurations in a way that they only
interact through balanced subsets.
Let us begin with a~$6$-cycle and phases~$(2k\pi/6+\alpha)_{k=1,\ldots,6}$.
Since~$\cos(2\pi/6) = 1/2$ the Jacobian matrix~\eqref{eq:jacobian} is
\begin{equation} \label{eq:circulant_matrix}
\small
\begin{pmatrix} 
-1 &\frac{1}{2} & 0 & 0 & 0 & \frac{1}{2} \\
\frac{1}{2} & -1 &  \frac{1}{2} & 0 &  0 &  0\\
0 & \frac{1}{2} & -1 & \frac{1}{2} & 0 & 0  \\
0 &  0  & \frac{1}{2} & -1  & \frac{1}{2} & 0 \\
0  & 0 & 0  & \frac{1}{2} & -1  & \frac{1}{2}\\
\frac{1}{2} & 0 & 0 & 0  & \frac{1}{2} & -1
\end{pmatrix}.
\normalsize
\end{equation}
Notice that it is independent of~$\alpha$.
Moreover, notice that the matrix is circulant. The eigenvalues of circulant
matrices are easy to compute~\cite{varga1954eigenvalues, gray2006toeplitz}.
In this case the characteristic polynomial is
\[
	p(\lambda) = \lambda (\lambda+1/2)^2 (\lambda+3/2)^2 (\lambda+2).
\]
Exactly one eigenvalue is zero and the others are strictly negative.
As~$\alpha$ can vary in~$\T$ we obtain a $1$-dimensional torus of stable equilibria.
The zero eigenvalue corresponds to the tangent space of this torus.

Notice that, in this configuration,
any pair of opposite vertices of the $6$-cycle is a balanced configuration.
This suggests how to connect several $6$-cycles together.

Consider~$d$ disjoint $6$-cycles and
add an edge between the first and fourth vertex of the first cycle with the first
and fourth vertex of every other cycle.
This amounts to adding~$4(d-1)$ edges.
Let~$\G_d$ denote the resulting graph.
The graph~$\G_2$ is the eye graph of Figure~\ref{fig:main} of the introduction.
The graph~$\G_5$ is given in Figure~\ref{fig:cycles}.

Let~$\alpha_1,\ldots,\alpha_d\in \T$ and
let~$(2k\pi/6+\alpha_i)_{k=1,\ldots,6}$ be the phases of the~$i$-th cycle.
The $6$-cycles only interact with each other through balanced pair of vertices.
By Lemma~\ref{lem:balanced} as~$\alpha_1,\ldots,\alpha_d$ vary in~$\T$
these configurations form a $d$-dimensional torus of equilibria~$\Gamma$.

In order to analyze stability, we look at Jacobian eigenvalues.
We want to show that there is an open set of~$(\alpha_1,\ldots,\alpha_d) \in \T^d$
for which exactly~$d$ Jacobian eigenvalues are~$0$ and the others strictly are negative.

In order to do so, we choose~$\alpha_1=0$ and~$\alpha_2=\ldots,\alpha_d = \pi/2$.
Every edge between two $6$-cycles
has phase difference~$\pi/2$, and since~$\cos(\pi/2)=0$ it follows that the Jacobian matrix
is a block diagonal matrix with~$d$ blocks all equal to~\eqref{eq:circulant_matrix}.
In particular the characteristic polynomial is
\[
	p(\lambda)^d = \lambda^d (\lambda+1/2)^{2d} (\lambda+3/2)^{2d} (\lambda+2)^d.
\]
Notice that~$0$ has multiplicity~$d$ and every other eigenvalue is strictly negative.

Since~$d$ is the dimension of the tangent space of~$\Gamma$, it follows that
there is a neighborhood of~$(0,\pi/2,\ldots,\pi/2) \in \T^d$ in which exactly~$d$
eigenvalues are zero and all the others are strictly negative.

Let us summarize what we have. There is a non-empty open subset~$V$ of~$\Gamma$
with the following properties: $V$ is a $d$-dimensional manifold,
every point in~$V$ is an equilibrium,
the stable manifold of every such equilibrium has codimension~$d$
and is orthogonal to~$V$.
The last statement holds since the Jacobian matrix~\eqref{eq:jacobian} is symmetric.

By Shoshitaishvili Theorem~\cite{liebscher2015bifurcation, shoshitaishvili1972bifurcations}
the stable manifolds form a foliation of the space near~$V$.
It follows that every~$\theta\in V$ is a Lyapunov stable equilibrium.
Indeed, a point near~$\theta$ lies in either the stable manifold of~$\theta$
or the stable manifold of a nearby equilibrium.
\end{proof}

\subsection{Stable Tori.}
In the proof of Theorem~\ref{thm:main} we obtain a torus of equilibria~$\Gamma$,
but only a proper subset is stable. Indeed, for~$d=2$ we can see in Figure~\ref{fig:eigenvalues}
that as the phase difference between the two cycles vary, an eigenvalue crosses zero;
stability is only guaranteed in a neighborhood of~$\pi/2$.
We now provide several graphs in which the set of stable equilibria contain
an entire $2$-dimensional torus.
Due to the size of the graphs, the eigenvalues
are computed numerically.

\begin{figure}[htbp]
\begin{subfigure}{0.33\columnwidth}
	\includegraphics[width=\linewidth]{noice.png}
    \caption{}
\end{subfigure} \qquad
\begin{subfigure}{0.33\columnwidth}
	\includegraphics[width=\linewidth]{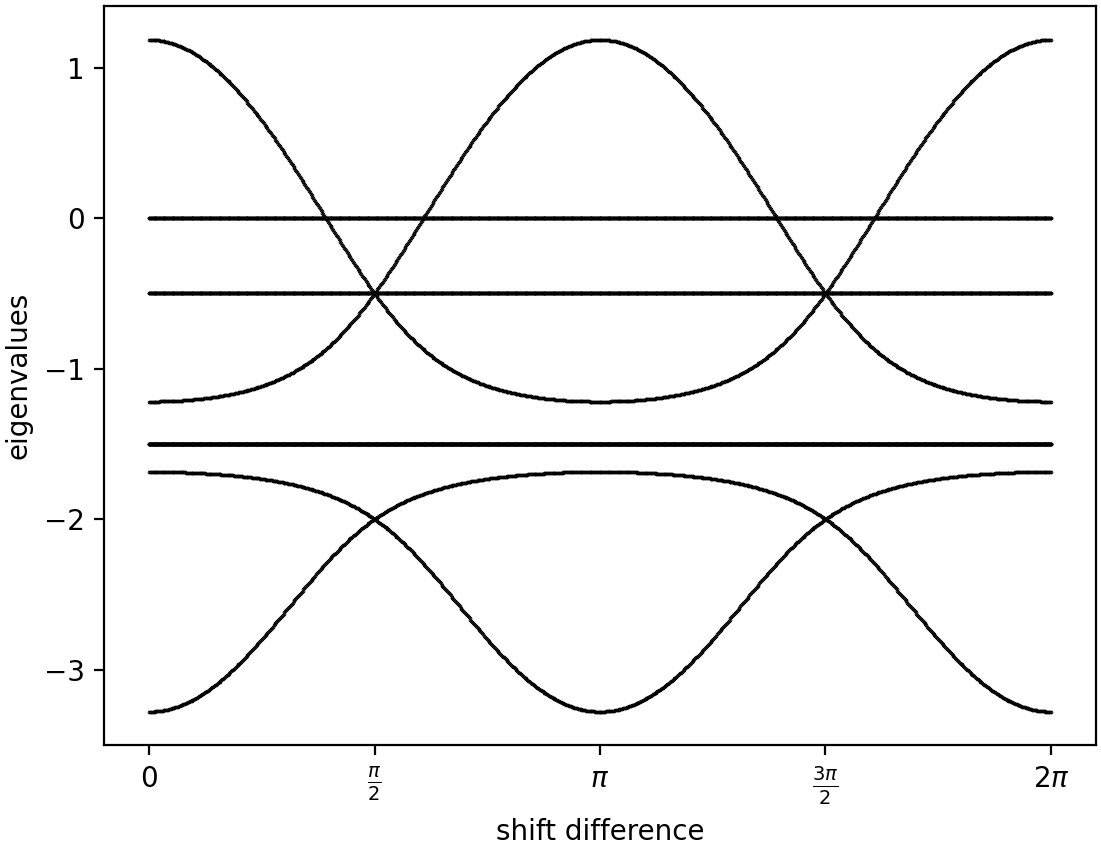}
    \caption{}
\end{subfigure}\\[0.1cm]
\begin{subfigure}{0.33\columnwidth}
	\includegraphics[width=\linewidth]{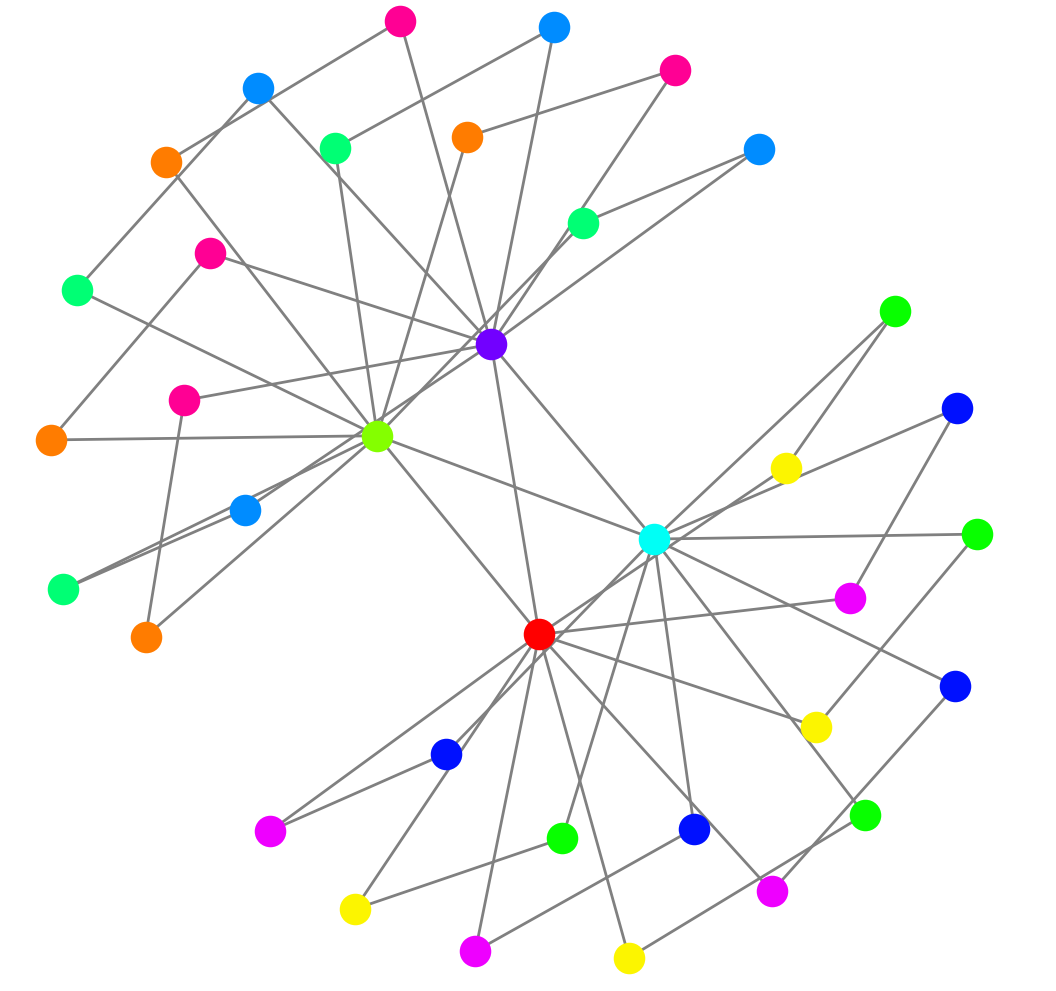}   
    \caption{}            
\end{subfigure} \qquad
\begin{subfigure}{0.33\columnwidth}
	\includegraphics[width=\linewidth]{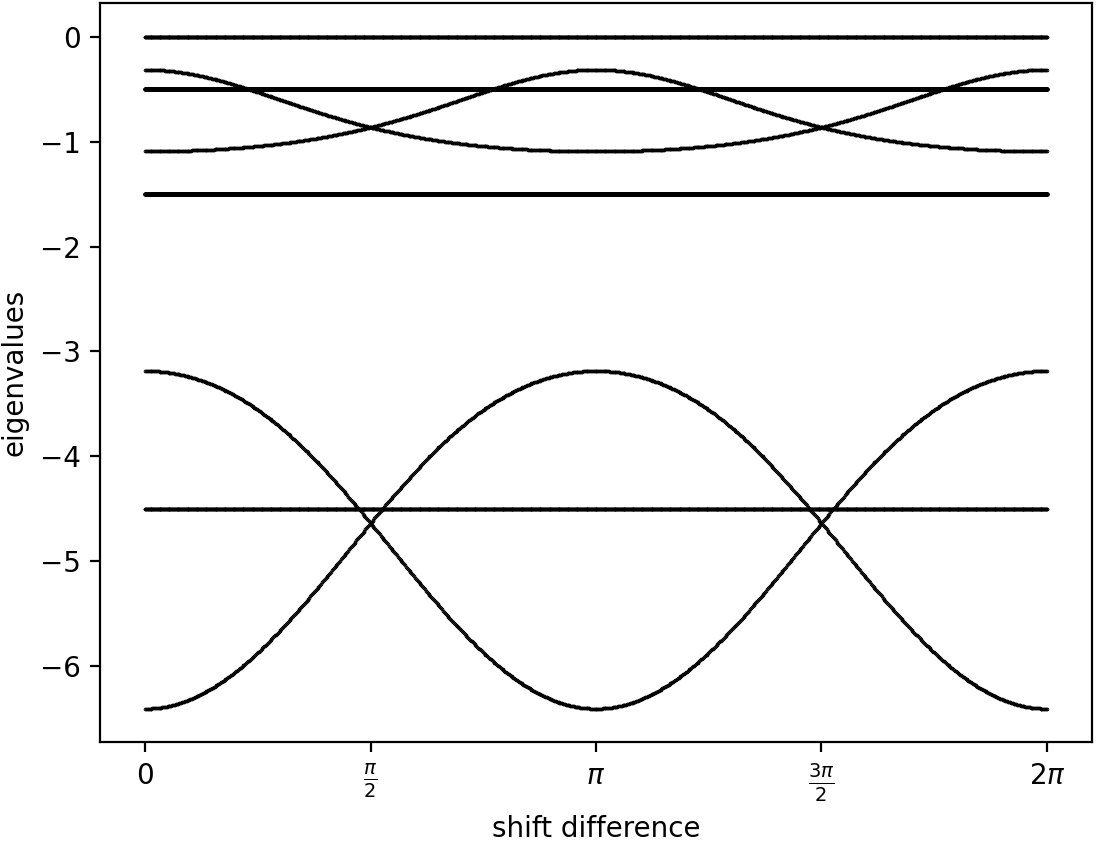}   
    \caption{}            
\end{subfigure}\\[0.1cm]
\begin{subfigure}{0.33\columnwidth}
	\includegraphics[width=\linewidth]{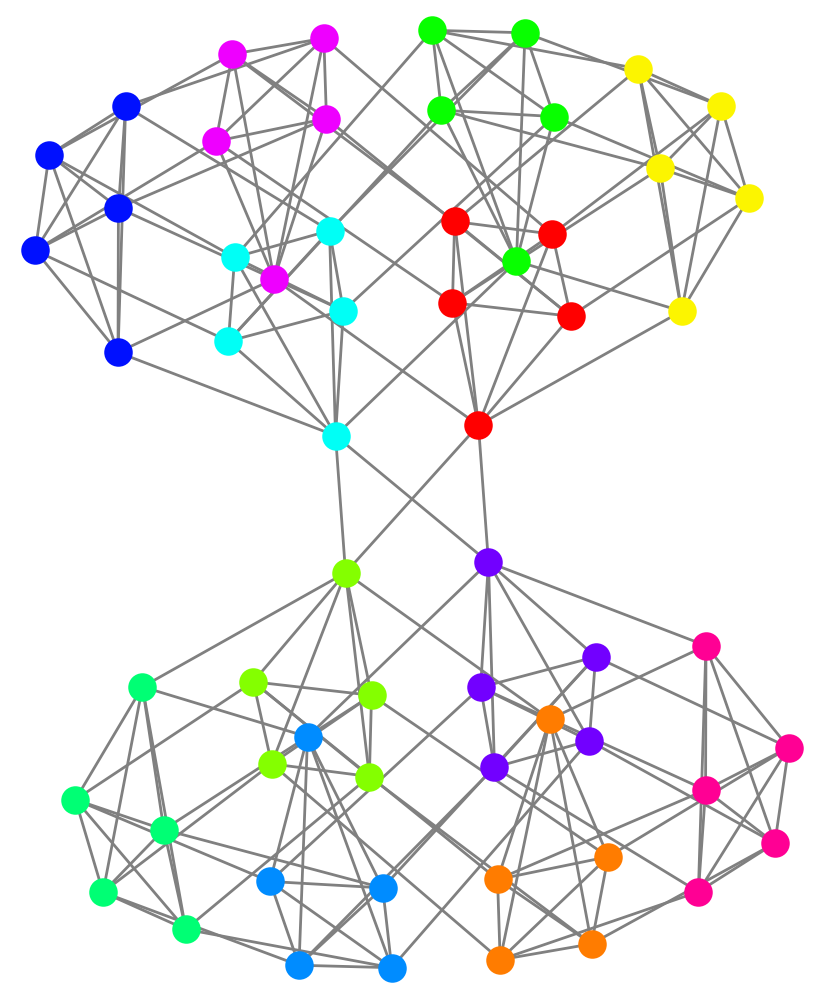}   
    \caption{}            
\end{subfigure} \qquad
\begin{subfigure}{0.33\columnwidth}
	\includegraphics[width=\linewidth]{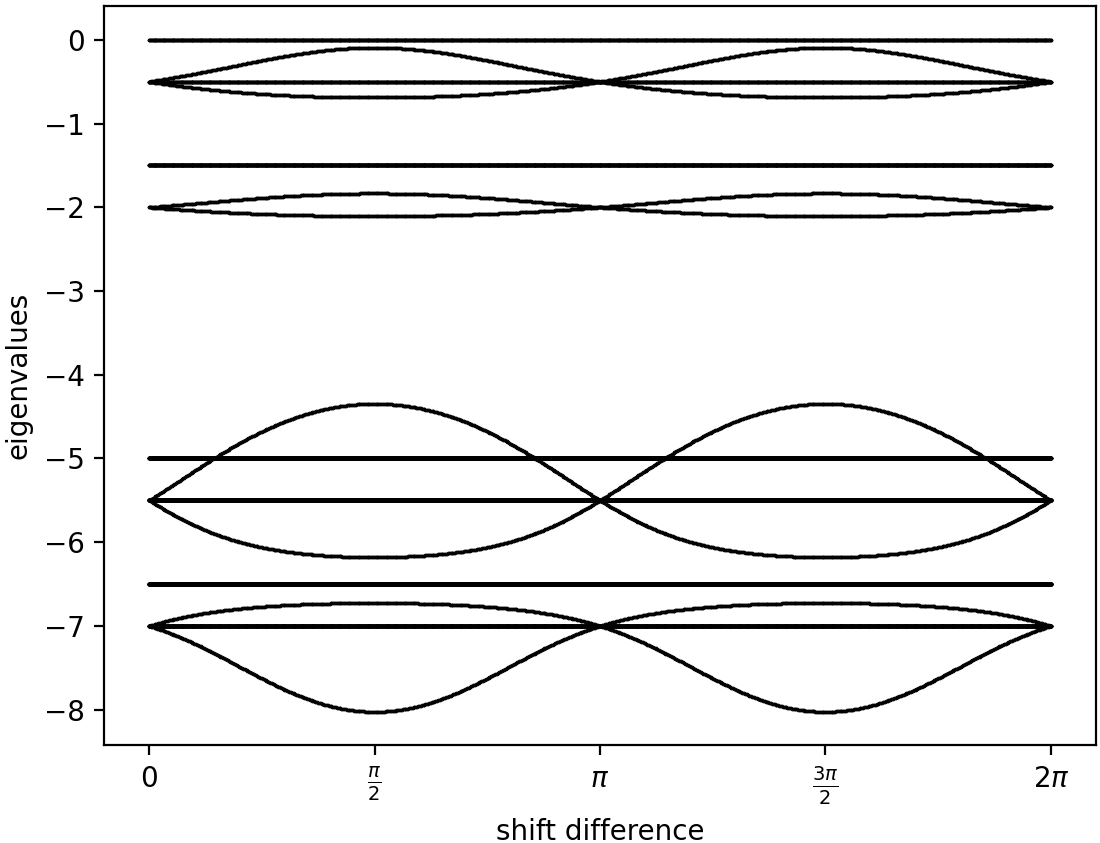}
    \caption{}            
\end{subfigure}\\[0.1cm]
\begin{subfigure}{0.33\columnwidth}
	\includegraphics[width=\linewidth]{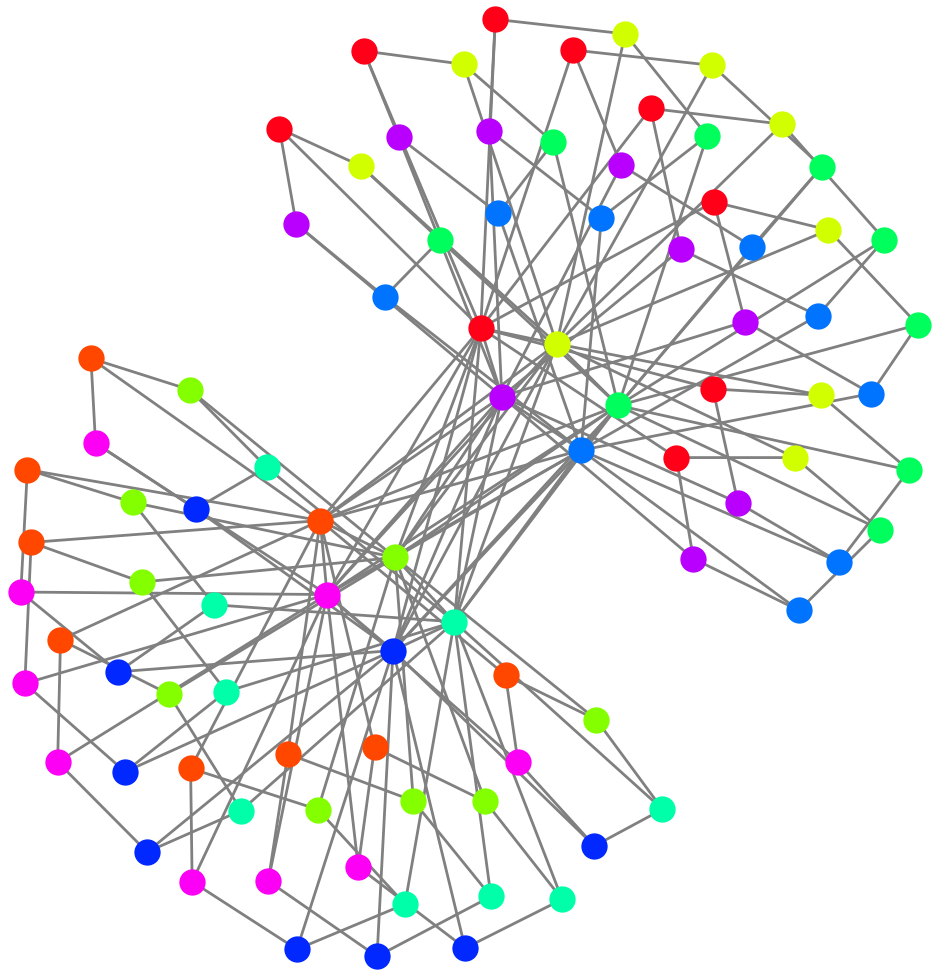}   
    \caption{}           
\end{subfigure} \qquad
\begin{subfigure}{0.33\columnwidth}
	\includegraphics[width=\linewidth]{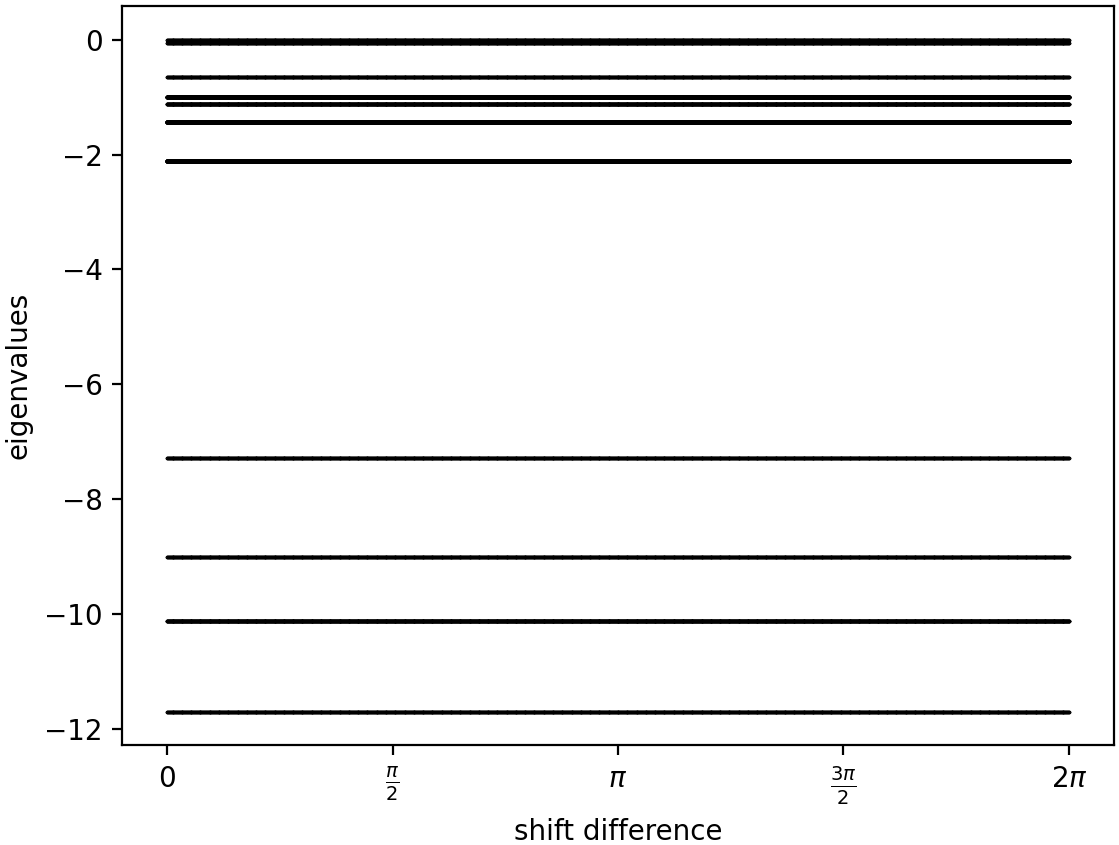}   
    \caption{}            
\end{subfigure}
\caption{On the left, graphs supporting $2$-dimensional manifold of stable equilibria.
On the right, eigenvalues as the phase-shift difference~$\beta-\alpha$ vary.}
\label{fig:eigenvalues}
\end{figure}

Let us denote by~$\H_{36}$, $\H_{60}$ and~$\H_{90}$ the graphs (C), (E), and~(G)
of Figure~\ref{fig:eigenvalues}, according to the number of vertices.
The vertices have been positioned so as to highlight the division into two subsets.
Each~$\H_{36}$, $\H_{60}$,~$\H_{90}$ supports a $2$-dimensional torus of stable equilibria.
The torus is obtained by phase-shifting the two parts.

Let us describe the configurations in more detail.
The graph $\H_{36}$ is obtained from the eye graph~$\G_2$ by taking of each
$6$-cycle~$3$ copies of itself and glueing them at the pair of vertices constituting
the balanced configuration.
One part has phases~$(2k\pi/6+\alpha)_{k=1,\ldots,6}$,
the other part~$(2k\pi/6+\beta)_{k=1,\ldots,6}$. Due to phase shift symmetry, stability
only depends on the phase difference~$\beta-\alpha$. As shown in Figure~\ref{fig:eigenvalues},
for every~$\alpha,\beta\in\T$ all non-zero eigenvalues are strictly negative.
As in the proof of Theorem~\ref{thm:main}, by Shoshitaishvili Theorem
it follows that for every~$\alpha,\beta\in \T$ this is a stable equilibrium.

Similarly~$\H_{60}$ is obtained from~$\G_2$
by substituting each vertex with a tetrahedron and connecting the tetrahedra in parallel
along the $6$-cycles.

By contrast, the graph~$\H_{90}$ is a bit different. It is obtained by fully connecting
two $5$-cycles with phases~$(2k\pi/5+\alpha)_{k=1,\ldots,5}$
and~$(2k\pi/5+\beta)_{k=1,\ldots,5}$ respectively. Notice that each cycle is a balanced
configuration. In order to make the configuration stable,
each cycle is copied~$8$ times and connected in parallel to its copies.
It is interesting to notice that in this case the eigenvalues do not depend on~$\alpha, \beta$,
see (H).

\subsection{Symmetry is not Necessary.}
Although the graphs discussed in this article are symmetric, symmetry is not necessary
for manifolds of stable equilibria.
Instead, it is useful for keeping the discussion as self-contained as possible,
as symmetric graphs can be described more easily in words.

We can enlarge a stable configuration of a graph
by connecting by an edge each vertex to an asymmetric graph whose phases
are synchronized with that vertex.
The enlarged configuration is stable. We can choose the auxiliary graphs so that
the enlarged graph is asymmetric.
If the original graph supports a $d$-dimensional manifolds
of stable equilibria, so does the enlarged one.

\subsection{The Geometry of Balanced Configurations.}
We are interested in the geometry of the set of balanced configurations of~$n$ vertices.

Two vertices are balanced if and only if the phases differ by~$\pi$. Therefore, balanced
configurations of two vertices form a $1$-dimensional torus.

There are two balanced
configurations with~$3$ vertices up to phase shift, corresponding to two non-equivalent
labeling of the vertices of an equilateral triangle,
thus leading to two disjoint $1$-dimensional tori.

Balanced configurations of~$4$ vertices are given by three non-equivalent
relabeling of the family~$(\alpha,\beta,\alpha+\pi,\beta+\pi)_{\alpha,\beta\in\T}$:
\begin{align*}
B_1&=\{(\alpha,\alpha+\pi,\beta,\beta+\pi) \mid \alpha,\beta\in\T\}, \\
B_2&=\{(\alpha,\beta,\alpha+\pi,\beta+\pi) \mid \alpha,\beta\in\T\}, \\
B_3&=\{(\alpha,\beta,\beta+\pi,\alpha+\pi) \mid \alpha,\beta\in\T\}.
\end{align*}
These $2$-dimensional tori intersect at aligned configurations.
Up to phase shift, the set of balanced configurations is homeomorphic to
three mutually tangent circles.

Until now, every balanced configuration is somewhat symmetric.
For~$n\geq 5$ the situation is far more complicated,
as the symmetry is lost.
Using Morse Theory it has been shown that,
up to phase shift, balanced configurations of~$5$
vertices form a surface of genus~$4$~\cite{havel1991use, kamiyama1992elementary}.
It is interesting that the separation between
the case~$n\leq 4$ and the case~$n\geq 5$ in terms of symmetry
becomes evident in terms of dynamics if the sine function is replaced by
a generic function~\cite[Proposition 2]{ashwin2016identical}.

In topology, balanced configurations appear in a different but equivalent form.
As noticed in~\cite{ashwin2008bifurcation}, we can associate to any balanced configuration
an equilateral polygon.
Let~$(\theta_k)_{k=1,\ldots,n}$ be a balanced configuration
and~$u_m = \sum_{k=1}^m e^{i\theta_k}$.
Then the sequence~$u_1,\ldots,u_n$ defines an equilateral polygon.
Here the term polygon must be taken somewhat loosely, since we allow vertices to
coincide and edges to intersect.

The space of equilateral polygons is studied, for example, in~\cite{kamiyama1996topology}.
Translating the results of~\cite{kamiyama1996topology} back into our language,
we can rephrase Theorem~A as follows:

\begin{theorem} [Y.Kamiyama] \label{thm:Y.Kamiyama}
Let~$n\geq 3$.
For an odd~$n$ the set of balanced configurations
is a smooth manifold of dimension~$n-2$.
For an even~$n$ the set of balanced configurations
is a manifold with singular points,
the generic dimension is~$n-2$ and the singular points are exactly the aligned configurations.
\end{theorem}

Recall that aligned configurations are those in which any two phases differ by~$0$ or~$\pi$.
Notice that configuration that are both balanced and aligned are only possible if~$n$ is even.
The presence of singular points for an even~$n$ is already evident
for~$n=4$, at the aforementioned tori~$B_1$, $B_2$, $B_3$ intersect.
We conjecture that for every~$n\geq 3$ the set of balanced configurations is a branched manifold
and can be written as a union of smooth manifolds.

As already noticed, balanced configurations correspond to equilateral polygons.
In the theory of
moduli spaces~\cite{kapovich1995moduli, mandini2014duistermaat, shimamoto2005spaces}
non-equilateral polygons are considered, usually with generic edge lengths.
This may be useful in better understanding the Kuramoto model on weighted graphs.


\section{Aligned Configurations and Complete Bipartite Graphs.} \label{sec:aligned}
In this section we discuss synchronized and aligned equilibria.
We are mainly interested in the case of complete bipartite graphs, where balanced
and aligned configurations characterize all the equilibria.
For our purposes, it is convenient to expand these notions to a proper subset of vertices.

\begin{definition} \label{def:aligned}
A configuration of~$\K$ is \emph{aligned} if the points~$\{e^{i\theta_k}\}_{k\in \K}$
belong to a line passing through the origin of~$\C$, that is,
any two phases differ by~$0$ or~$\pi$. If all the phases are equal
we say that~$\K$ is~\emph{synchronized}.
\end{definition}

Unlike balanced configurations,
the aligned configuration of~$\V(\G)$ is an equilibrium for every~$\G$.
We refer to these configurations as aligned equilibria.
There are~$2^{n-1}$ non-equivalent aligned equilibria,
each corresponding to a~$1$-dimensional torus.
As we will see in a later section, some of them may be part of larger component.
Notice that there is only one synchronized equilibrium up to phase-shift.

\begin{proposition}
Let~$\G$ be connected. Then the only stable aligned equilibria of~$\G$
are the synchronized equilibria.
\end{proposition}
\begin{proof}
Let~$\theta$ be an aligned equilibrium. The vertices~$\V(\G)$ can be partitioned into
two subsets~$\J$ and~$\K$ such that every vertex in~$\J$ have phase~$\alpha$
and every vertex in~$\K$ have phase~$\alpha+\pi$, for some~$\alpha\in\T$.
Let~$\E(\J,\K)$ denote
the set of edges between~$\J$ and~$\K$ and~$\abs{\E(\J,\K)}$ the cardinality of this set.

Suppose that~$\theta$ is not a synchronized equilibrium, then~$\J$ and~$\K$ are non-empty.
Since~$\G$ is connected the set~$\E(\J,\K)$ is non-empty.
Perturb~$\theta$ as~$\tilde\theta_k = \theta_k + x$
for~$k\in \K$ and~$\tilde\theta_j = \theta_j$ for~$j\in \J$. We claim that
that the energy~\eqref{eq:energy} decreases in the direction of the perturbation.
Indeed
\begin{align*}
	E(\theta) - E(\tilde\theta) & =
		\sum_{jk \in \E(\G)}
		\cos(\tilde\theta_j - \tilde\theta_k) - \cos(\theta_j - \theta_k) \\
		& = \sum_{jk \in \E(\J,\K)}
		\cos(\tilde\theta_j - \tilde\theta_k) - \cos(\theta_j - \theta_k) \\
		& = \abs{\E(\J,\K)} (\cos(\pi + x) - \cos(\pi))
\end{align*}
and since~$\abs{\mathcal E(\J,\K)}$ is positive
the difference~$E(\theta) - E(\tilde\theta)$ is positive for every~$x$ small enough.
The equilibrium~$\theta$ is not a local minimum, and thus it is not stable.

It remains to prove that the synchronized equilibria are stable. This follows from the fact that,
for every graph~$\G$, synchronized equilibria
are exactly the global minima of the energy~\eqref{eq:energy}.
\end{proof}

\subsection{Complete Bipartite Graphs.} \label{sec:bipartite}
In order to show how balanced configurations and aligned configurations may interplay,
we discuss complete bipartite graphs.
In~\cite{canale2009characterization} it is shown that the synchronized configurations
are the only stable equilibria in complete bipartite graph.
Here we complete the analysis by classifying all the possible equilibria, with a focus
on their geometric structure.
In particular, we see that balanced configurations lead to manifolds of unstable equilibria
of arbitrarily large dimension.

\begin{proposition}~\label{prop:bipartite}
Let~$\G$ be complete bipartite with non-empty parts~$\J$ and~$\K$.
The equilibria of~$\G$ have the following form:
\begin{enumerate} [label = (\roman*)]
\item $\theta_\J$~is balanced,~$\theta_\K$ is balanced; \label{prop:bipartite:1}
\item $\theta_\J$~is not balanced, $\theta_\K$ is balanced
and aligned to~$\sum_{j\in\J} e^{i\theta_j}$; \label{prop:bipartite:2}
\item $\theta_\K$~is not balanced, $\theta_\J$ is balanced
and aligned to~$\sum_{j\in\K} e^{i\theta_k}$; \label{prop:bipartite:3}
\item $\theta$ is aligned. \label{prop:bipartite:4}
\end{enumerate}
\end{proposition}

\begin{proof}
It is easy to see that every configuration of~$\V(\G)$ satisfying
any of the conditions~\ref{prop:bipartite:1} through~\ref{prop:bipartite:4}
is an equilibrium.

Conversely, let~$\theta$ be an equilibrium. We prove that~$\theta$ satisfies at least one
of the conditions. We have
\begin{align}
	& \Im \p{\sum_{k\in \K} e^{i\theta_k} e^{-i\theta_j}} = 0,
		\qquad \text{for all } j\in\J; \label{prop:bipartite:5} \\
	& \Im \p{\sum_{j\in \J} e^{i\theta_j} e^{-i\theta_k}} = 0,
		\qquad \text{for all } k\in\K . \label{prop:bipartite:6}
\end{align}

If~$\J$ and~$\K$ are both balanced we are in case~\ref{prop:bipartite:1}.
Suppose that~$\J$ is not balanced. Due to the phase-shift symmetry, we can assume
that~$\sum_{j\in \J} e^{i\theta_j}$ is
real and strictly positive. From~\eqref{prop:bipartite:6}
it follows that~$\theta_k\in\{0,\pi\}$ for every~$k\in \K$.
Then from~\eqref{prop:bipartite:5} it follows that either~$\K$ has an even number of vertices,
half with phase~$0$ and half with phase~$\pi$, and we are in case~\ref{prop:bipartite:2},
or~$\theta_j\in\{0,\pi\}$ for every~$j\in \J$, and we are in case~\ref{prop:bipartite:4}.
Similarly, if~$\K$ is not balanced
we are in case~\ref{prop:bipartite:3} or in case~\ref{prop:bipartite:4}.
\end{proof}

Let~$\abs{\J} = n$ and $\abs{\K} = m$. Suppose that~$n,m\geq 3$.
By Theorem~\ref{thm:Y.Kamiyama}
the set of equilibria of the form~\ref{prop:bipartite:1} has dimension~$(n-2)(m-2)$.
The set of equilibria of the form~\ref{prop:bipartite:2} has dimension~$n$ if~$m$ is even
and is empty otherwise. Similarly,
the set of equilibria of the form~\ref{prop:bipartite:3} has dimension~$m$ if~$n$ is even
and is empty otherwise. The aligned configurations~\ref{prop:bipartite:2}
form a set of dimension~$1$.

Notice that the cases~\ref{prop:bipartite:1} through~\ref{prop:bipartite:4} are not disjoint
and the conditions~\ref{prop:bipartite:2} and~\ref{prop:bipartite:3} are not closed.
Therefore, in order to obtain the decomposition in irreducible components
some more work is needed. In the next section we give an explicit example.


\section{Manifolds Connected by Heteroclinic Orbits}~\label{sec:heteroclinic}
A \emph{heteroclinic orbit} is a solution joining two equilibria.
We know from Section~\ref{sec:tools} that every solution of~\eqref{eq:main}
is either an equilibrium or a heteroclinic orbit
connecting two irreducible sets of equilibria.
Given a graph, we are led to the following program:
\begin{itemize}
\item computing the irreducible decomposition
of the set of equilibria;
\item establishing how the components are connected by solutions.
\end{itemize}
As we will see by two examples, even for small graphs this program may lead to some
intricate topological structures.

\subsection{Complete graph with $4$ vertices.}
Consider a complete graphs with vertices $\V(\G) = \{1, 2, 3, 4\}$.
The only equilibria are
the balanced configurations of~$\V(\G)$ and the aligned
configurations of~$\V(\G)$. As explained in Section~\ref{sec:balanced}
the balanced configurations of four vertices are arranged into three
$2$-dimensional tori:
\begin{align*}
B_1&=\{(\alpha,\alpha+\pi,\beta,\beta+\pi) \mid \alpha,\beta\in\T\}, \\
B_2&=\{(\alpha,\beta,\alpha+\pi,\beta+\pi) \mid \alpha,\beta\in\T\}, \\
B_3&=\{(\alpha,\beta,\beta+\pi,\alpha+\pi) \mid \alpha,\beta\in\T\}.
\end{align*}
These tori mutually intersect at aligned balanced configurations. Up to phase-shift
the balanced configurations form three mutually tangent circles.
There are only three aligned balanced configurations up to phase-shift.
The energy of the balanced configurations is~$E=8$, which is the maximum of the system.

The aligned configurations that are not balanced are
\begin{align*}
& A_1 = \{(\alpha, \alpha+\pi, \alpha+\pi, \alpha+\pi) \mid \alpha \in \T\}, \\
& A_2 = \{(\alpha +\pi, \alpha, \alpha+\pi, \alpha+\pi) \mid \alpha \in \T\}, \\
& A_3 = \{(\alpha +\pi, \alpha+\pi, \alpha, \alpha+\pi) \mid \alpha \in \T\}, \\
& A_4 = \{(\alpha +\pi, \alpha+\pi, \alpha+\pi, \alpha) \mid \alpha \in \T\},
\end{align*}
with energy~$E=6$, and the synchronized state
\[
	S = \{(\alpha, \alpha, \alpha, \alpha) \mid \alpha \in \T\}
\]
with energy~$E=0$. By perturbing the phase of one vertex,
it is easy to see that the equilibria in~$A_1$, $A_2$, $A_3$, and~$A_4$ are not
local maxima nor local minima.

A heteroclinic connection is only possible from a set with higher
energy to a set with lower energy. Numerically, we see that any connection
allowed by the energy actually exists. The existence of an orbit between~$A_i$
and~$S$ is trivial since~$S$ is the global minimum of
the energy. In order to establish the existence of an orbit between~$A_i$
and~$B_j$,
we perturbed an equilibrium of~$A_i$ and applied the Runge–Kutta method RK4 with reversed time.
We obtain the heteroclinic structure
of Figure~\ref{fig:heteroclinic_complete}.

\begin{figure} [h]
\centering
\begin{tikzpicture} [scale=1]
\draw[] (-2,0) ellipse (2 and 1);
\draw[] (2,0) ellipse (2 and 1);
\draw[] (0,0) ellipse (4 and 1.6);
\fill (-3,-3) circle (0.05);
\fill (-1,-3) circle (0.05);
\fill (+1,-3) circle (0.05);
\fill (+3,-3) circle (0.05);
\fill (0,-4) circle (0.05);
\fill (0,0) circle (0.05);
\fill (-4,0) circle (0.05);
\fill (+4,0) circle (0.05);
\draw[->, dashed, shorten >= 5pt, shorten <= 0pt, line width = 0.1]
	(-2,-1) to [] (-3,-3);
\draw[->, dashed, shorten >= 3pt, shorten <= 0pt, line width = 0.1]
	(-2,-1) to [] (-1,-3);
\draw[->, dashed, shorten >= 2pt, shorten <= 0pt, line width = 0.1]
	(-2,-1) to [] (+1,-3);
\draw[->, dashed, shorten >= 3pt, shorten <= 0pt, line width = 0.1]
	(-2,-1) to [] (+3,-3);
\draw[->, dashed, shorten >= 3pt, shorten <= 0pt, line width = 0.1]
	(0,-1.6) to [] (-3,-3);
\draw[->, dashed, shorten >= 5pt, shorten <= 0pt, line width = 0.1]
	(0,-1.6) to [] (-1,-3);
\draw[->, dashed, shorten >= 5pt, shorten <= 0pt, line width = 0.1]
	(0,-1.6) to [] (+1,-3);
\draw[->, dashed, shorten >= 3pt, shorten <= 0pt, line width = 0.1]
	(0,-1.6) to [] (+3,-3);
\draw[->, dashed, shorten >= 3pt, shorten <= 0pt, line width = 0.1]
	(+2,-1) to [] (-3,-3);
\draw[->, dashed, shorten >= 2pt, shorten <= 0pt, line width = 0.1]
	(+2,-1) to [] (-1,-3);
\draw[->, dashed, shorten >= 3pt, shorten <= 0pt, line width = 0.1]
	(+2,-1) to [] (+1,-3);
\draw[->, dashed, shorten >= 5pt, shorten <= 0pt, line width = 0.1]
	(+2,-1) to [] (+3,-3);
\draw[->, dashed, shorten >= 3pt, shorten <= 0pt, line width = 0.1]
	(-2,-1) to [out = -90-45, in = 180, distance=120pt] (0,-4);
\draw[->, dashed, shorten >= 3pt, shorten <= 0pt, line width = 0.1]
	(0,-1.6) to [] (0,-4);
\draw[->, dashed, shorten >= 3pt, shorten <= 0pt, line width = 0.1]
	(+2,-1) to [out = -45, in = 0, distance=120pt] (0,-4);
\draw[->, dashed, shorten >= 3pt, shorten <= 3pt, line width = 0.1]
	(-3,-3) to [out = -45, in = 90, distance=10pt] (0,-4);
\draw[->, dashed, shorten >= 3pt, shorten <= 3pt, line width = 0.1]
	(-1,-3) to [out = -45, in = 90, distance=10pt] (0,-4);
\draw[->, dashed, shorten >= 3pt, shorten <= 3pt, line width = 0.1]
	(+1,-3) to [out = -90-45, in = 90, distance=10pt] (0,-4);
\draw[->, dashed, shorten >= 3pt, shorten <= 3pt, line width = 0.1]
	(+3,-3) to [out = -90-45, in = 90, distance=10pt] (0,-4);
\node [fill=white] at (-2, -1+0.1) {$B_1$};
\node [fill=white] at (+0, -1.6+0.2) {$B_2$};
\node [fill=white] at (+2, -1+0.1) {$B_3$};
\node [] at (-3-0.3, -3) {$A_1$};
\node [] at (-1-0.3, -3) {$A_2$};
\node [] at (+1+0.3, -3) {$A_3$};
\node [] at (+3+0.3, -3) {$A_3$};
\node [] at (0, -4-0.3) {$S$};
\end{tikzpicture}
\caption{Topology of the set of equilibria up to phase shift of the complete
graph with $4$~vertices.
A dotted arrow represents the existence of a heteroclinic orbit.}
\label{fig:heteroclinic_complete}
\end{figure}
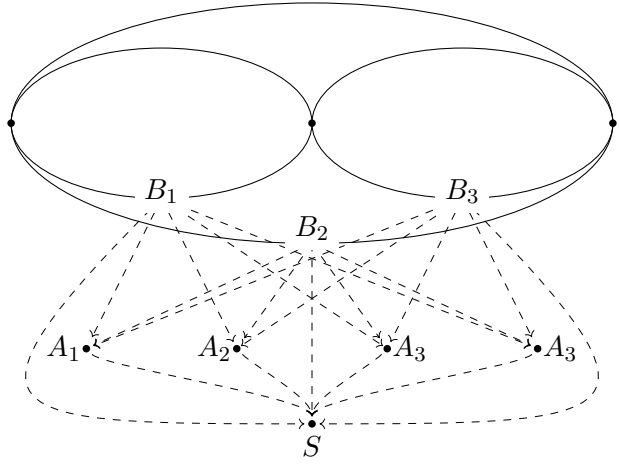

\subsection{Cycle with $4$ vertices.}
A cycle graph with vertices~$1$, $2$, $3$, $4$
is a complete bipartite graph with parts~$\J=\{1,3\}$
and~$\K=\{2,4\}$. Therefore, we can apply Proposition~\ref{prop:bipartite}.
Equilibria of the form~\ref{prop:bipartite:1}, in which each part is balanced,
constitute a $2$-dimensional torus
\[
	B_2 = \{(\alpha,\beta,\alpha+\pi,\beta+\pi) \mid \alpha,\beta\in \T\}
\]
with energy~$E=4$.
Equilibria of the form~\ref{prop:bipartite:2} and~\ref{prop:bipartite:3},
in which one part is generic and the other is balanced and aligned give
the $2$-dimensional tori
\begin{align*}
	C_1 &= \{(\beta-\delta, \beta, \beta+\delta , \beta+\pi)
		\mid \beta, \delta\in \T \}, \\
	C_2 &= \{(\alpha, \alpha+\delta , \alpha+\pi, \alpha-\delta)
		\mid \alpha, \delta\in \T \},
\end{align*}
with energy~$E=4$.
Notice that the conditions~\ref{prop:bipartite:2} and~\ref{prop:bipartite:3} are not closed,
here we are taking closures. It is easy to see that the equilibria in~$B_2$, $C_1$ and~$C_2$
are not local maxima nor local minima.

Notice that~$B_2$, $C_1$ and~$C_2$ contain two aligned equilibria each.
The remaining aligned equilibria are
\begin{align*}
	A_5 &= \{ (\alpha, \alpha+\pi, \alpha, \alpha+\pi) \mid \alpha \in \T\}, \\
	S &= \{ (\alpha, \alpha, \alpha, \alpha) \mid \alpha \in \T\},
\end{align*}
with energy~$E=8$ and~$E=0$ respectively.
Therefore, we obtain the heteroclinic network of Figure~\ref{fig:heteroclinic_cycle}.

Notice that~$B_2$, $C_1$ and~$C_2$ intersect in two $1$-dimensional tori.
The equilibria in the intersection have the property that all Jacobian eigenvalues are zero.
Equilibria with this property, known as completely degenerate, correspond to the Eulerian
circuits in the graph~\cite{sclosa2021completely}. In this case, the intersections
correspond to the two ways of traveling around the cycle.

\begin{figure} [h]
\centering
\begin{tikzpicture} [scale=1]
\draw (0,0) ellipse (1.5 and 1.5);
\draw (0,0) ellipse (2.5 and 1.5);
\draw (0,0) ellipse (3.5 and 1.5);
\fill (0,+2.5) circle (0.05);
\fill (0,-2.5) circle (0.05);
\fill (0,+1.5) circle (0.05);
\fill (0,-1.5) circle (0.05);
\draw[->, dashed, shorten >= 10pt, shorten <= 5pt, line width = 0.1]
	(0,2.5) to [out=-10, in=80, distance=30pt] (1.5,0);
\draw[->, dashed, shorten >= 10pt, shorten <= 5pt, line width = 0.1]
	(0,2.5) to [out=-10, in=80, distance=30pt] (2.5,0);
\draw[->, dashed, shorten >= 10pt, shorten <= 5pt, line width = 0.1]
	(0,2.5) to [out=-10, in=80, distance=30pt] (3.5,0);
\draw[->, dashed, shorten >= 5pt, shorten <= 5pt, line width = 0.1]
	(0,2.5) to [out=-170, in=170, distance=150pt] (0,-2.5);
\draw[<-, dashed, shorten <= 5pt, line width = 0.1]
	(0,-2.5) to [out=10, in=-90, distance=30pt] (1.5,0);
\draw[<-, dashed, shorten <= 5pt, line width = 0.1]
	(0,-2.5) to [out=10, in=-90, distance=30pt] (2.5,0);
\draw[<-, dashed, shorten <= 5pt, line width = 0.1]
	(0,-2.5) to [out=10, in=-90, distance=30pt] (3.5,0);
\node at (0.1, 2.8) {$A_5$};
\node at (0.1, -2.9) {$S$};
\node [{fill=white}] at (3.5, 0) {$C_2$};
\node [{fill=white}] at (2.5, 0) {$C_1$};
\node [{fill=white}] at (1.5, 0) {$B_2$};
\end{tikzpicture}
\caption{Topology of the set of equilibria up to phase shift of the cycle graph with $4$~vertices.
A dotted arrow represents the existence of a heteroclinic orbit.}
\label{fig:heteroclinic_cycle}
\end{figure}
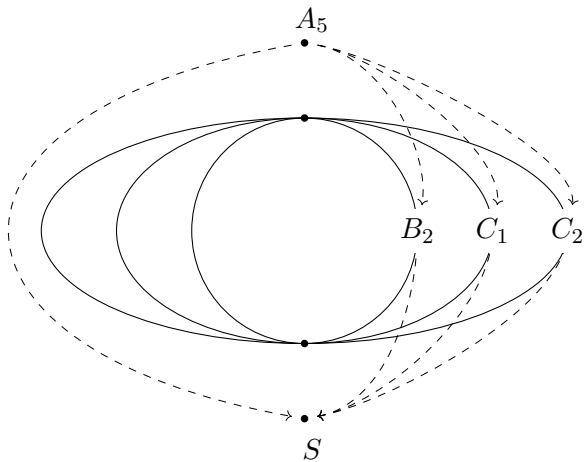

\bibliographystyle{siam}
\bibliography{refs}

\begin{thebibliography}{10}

\bibitem{Absil2006}
{\sc P.~A. Absil and K.~Kurdyka}, {\em {On the stable equilibrium points of
  gradient systems}}, Systems and Control Letters, 55 (2006), pp.~573--577.

\bibitem{ashwin2016identical}
{\sc P.~Ashwin, C.~Bick, and O.~Burylko}, {\em Identical phase oscillator
  networks: Bifurcations, symmetry and reversibility for generalized coupling},
  Frontiers in Applied Mathematics and Statistics, 2 (2016), p.~7.

\bibitem{ashwin2008bifurcation}
{\sc P.~Ashwin, O.~Burylko, and Y.~Maistrenko}, {\em Bifurcation to
  heteroclinic cycles and sensitivity in three and four coupled phase
  oscillators}, Physica D: Nonlinear Phenomena, 237 (2008), pp.~454--466.

\bibitem{baillieul1982geometric}
{\sc J.~Baillieul and C.~Byrnes}, {\em Geometric critical point analysis of
  lossless power system models}, IEEE Transactions on Circuits and Systems, 29
  (1982), pp.~724--737.

\bibitem{bick2011chaos}
{\sc C.~Bick, M.~Timme, D.~Paulikat, D.~Rathlev, and P.~Ashwin}, {\em Chaos in
  symmetric phase oscillator networks}, Physical review letters, 107 (2011),
  p.~244101.

\bibitem{brown2003globally}
{\sc E.~Brown, P.~Holmes, and J.~Moehlis}, {\em Globally coupled oscillator
  networks}, in Perspectives and Problems in Nolinear Science, Springer, 2003,
  pp.~183--215.

\bibitem{canale2009}
{\sc E.~Canale and P.~Monzon}, {\em Global properties of kuramoto
  bidirectionally coupled oscillators in a ring structure}, in 2009 IEEE
  Control Applications,(CCA) \& Intelligent Control,(ISIC), IEEE, 2009,
  pp.~183--188.

\bibitem{canale2009characterization}
{\sc E.~A. Canale and P.~A. Monzon}, {\em On the characterization of families
  of synchronizing graphs for kuramoto coupled oscillators}, IFAC Proceedings
  Volumes, 42 (2009), pp.~42--47.

\bibitem{Chen2018}
{\sc T.~Chen, R.~Davis, and D.~Mehta}, {\em {Counting Equilibria of the
  Kuramoto Model Using Birationally Invariant Intersection Index}}, SIAM
  Journal on Applied Algebra and Geometry, 2 (2018), pp.~489--507.

\bibitem{delabays2017multistability}
{\sc R.~Delabays, T.~Coletta, and P.~Jacquod}, {\em Multistability of
  phase-locking in equal-frequency kuramoto models on planar graphs}, Journal
  of Mathematical Physics, 58 (2017), p.~032703.

\bibitem{DeVille2016}
{\sc L.~DeVille and B.~Ermentrout}, {\em {Phase-locked patterns of the Kuramoto
  model on 3-regular graphs}}, Chaos, 26 (2016), pp.~1--11.

\bibitem{godsil2001algebraic}
{\sc C.~Godsil and G.~F. Royle}, {\em Algebraic graph theory}, vol.~207,
  Springer Science \& Business Media, 2001.

\bibitem{gray2006toeplitz}
{\sc R.~M. Gray}, {\em Toeplitz and circulant matrices: A review},  (2006).

\bibitem{hartshorne2013algebraic}
{\sc R.~Hartshorne}, {\em Algebraic geometry}, vol.~52, Springer Science \&
  Business Media, 2013.

\bibitem{havel1991use}
{\sc T.~Havel}, {\em The use of distances as coordinates in computer-aides
  proofs of theorems in euclidean geometry}, J. Symbolic Comput., 11 (1991),
  pp.~579--593.

\bibitem{hendrickx2012convergence}
{\sc J.~M. Hendrickx and J.~N. Tsitsiklis}, {\em Convergence of type-symmetric
  and cut-balanced consensus seeking systems}, IEEE Transactions on Automatic
  Control, 58 (2012), pp.~214--218.

\bibitem{jadbabaie2004stability}
{\sc A.~Jadbabaie, N.~Motee, and M.~Barahona}, {\em On the stability of the
  kuramoto model of coupled nonlinear oscillators}, in Proceedings of the 2004
  American Control Conference, vol.~5, IEEE, 2004, pp.~4296--4301.

\bibitem{kamiyama1992elementary}
{\sc Y.~Kamiyama}, {\em An elementary proof of a theorem of tf havel}, Ryukyu
  Math. J, 5 (1992), pp.~7--12.

\bibitem{kamiyama1996topology}
\leavevmode\vrule height 2pt depth -1.6pt width 23pt, {\em Topology of
  equilateral polygon linkages}, Topology and its Applications, 68 (1996),
  pp.~13--31.

\bibitem{kapovich1995moduli}
{\sc M.~Kapovich and J.~Millson}, {\em On the moduli space of polygons in the
  euclidean plane}, Journal of Differential Geometry, 42 (1995), pp.~133--164.

\bibitem{kapovich1996symplectic}
{\sc M.~Kapovich and J.~J. Millson}, {\em The symplectic geometry of polygons
  in euclidean space}, Journal of Differential Geometry, 44 (1996),
  pp.~479--513.

\bibitem{Kassabov2021}
{\sc M.~Kassabov, S.~H. Strogatz, and A.~Townsend}, {\em {Sufficiently dense
  Kuramoto networks are globally synchronizing}}, Chaos: An Interdisciplinary
  Journal of Nonlinear Science, 31 (2021), p.~073135.

\bibitem{kopell1995anti}
{\sc N.~Kopell and D.~Somers}, {\em Anti-phase solutions in relaxation
  oscillators coupled through excitatory interactions}, Journal of mathematical
  biology, 33 (1995), pp.~261--280.

\bibitem{kuramoto1987statistical}
{\sc Y.~Kuramoto and I.~Nishikawa}, {\em Statistical macrodynamics of large
  dynamical systems. case of a phase transition in oscillator communities},
  Journal of Statistical Physics, 49 (1987), pp.~569--605.

\bibitem{lang2019introduction}
{\sc S.~Lang}, {\em Introduction to algebraic geometry}, Courier Dover
  Publications, 2019.

\bibitem{liebscher2015bifurcation}
{\sc S.~Liebscher}, {\em Bifurcation without parameters}, vol.~526, Springer,
  2015.

\bibitem{Ling2019}
{\sc S.~Ling, R.~Xu, and A.~S. Bandeira}, {\em {On the landscape of
  synchronization networks: A perspective from nonconvex optimization}}, SIAM
  Journal on Optimization, 29 (2019), pp.~1879--1907.

\bibitem{lu2020}
{\sc J.~Lu and S.~Steinerberger}, {\em Synchronization of kuramoto oscillators
  in dense networks}, Nonlinearity, 33 (2020), p.~5905.

\bibitem{mandini2014duistermaat}
{\sc A.~Mandini}, {\em The duistermaat-heckman formula and the cohomology of
  moduli spaces of polygons}, Journal of Symplectic Geometry, 12 (2014),
  pp.~171--213.

\bibitem{markdahl2021counterexamples}
{\sc J.~Markdahl}, {\em Counterexamples in synchronization: pathologies of
  consensus seeking gradient descent flows on surfaces}, Automatica, 134
  (2021), p.~109945.

\bibitem{martens2016chimera}
{\sc E.~A. Martens, C.~Bick, and M.~J. Panaggio}, {\em Chimera states in two
  populations with heterogeneous phase-lag}, Chaos: An Interdisciplinary
  Journal of Nonlinear Science, 26 (2016), p.~094819.

\bibitem{mirollo2012asymptotic}
{\sc R.~E. Mirollo}, {\em The asymptotic behavior of the order parameter for
  the infinite-n kuramoto model}, Chaos: An Interdisciplinary Journal of
  Nonlinear Science, 22 (2012), p.~043118.

\bibitem{park2016weakly}
{\sc Y.~Park and B.~Ermentrout}, {\em Weakly coupled oscillators in a slowly
  varying world}, Journal of computational neuroscience, 40 (2016),
  pp.~269--281.

\bibitem{pazo2005thermodynamic}
{\sc D.~Paz{\'o}}, {\em Thermodynamic limit of the first-order phase transition
  in the kuramoto model}, Physical Review E, 72 (2005), p.~046211.

\bibitem{ren2005consensus}
{\sc W.~Ren and R.~W. Beard}, {\em Consensus seeking in multiagent systems
  under dynamically changing interaction topologies}, IEEE Transactions on
  automatic control, 50 (2005), pp.~655--661.

\bibitem{sclosa2021completely}
{\sc D.~Sclosa}, {\em Completely degenerate equilibria of the kuramoto model on
  networks}, arXiv preprint arXiv:2112.12034,  (2021).

\bibitem{shafarevich1994basic}
{\sc I.~R. Shafarevich and M.~Reid}, {\em Basic algebraic geometry}, vol.~2,
  Springer, 1994.

\bibitem{shimamoto2005spaces}
{\sc D.~Shimamoto and C.~Vanderwaart}, {\em Spaces of polygons in the plane and
  morse theory}, The American Mathematical Monthly, 112 (2005), pp.~289--310.

\bibitem{shoshitaishvili1972bifurcations}
{\sc A.~N. Shoshitaishvili}, {\em Bifurcations of topological type at singular
  points of parametrized vector fields}, Funktsional'nyi Analiz i ego
  Prilozheniya, 6 (1972), pp.~97--98.

\bibitem{sokolov2019sync}
{\sc Y.~Sokolov and G.~B. Ermentrout}, {\em When is sync globally stable in
  sparse networks of identical kuramoto oscillators?}, Physica A: Statistical
  Mechanics and its Applications, 533 (2019), p.~122070.

\bibitem{Taylor2012}
{\sc R.~Taylor}, {\em There is no non-zero stable fixed point for dense
  networks in the homogeneous kuramoto model}, Journal of Physics A:
  Mathematical and Theoretical, 45 (2012), p.~055102.

\bibitem{van1993lyapunov}
{\sc J.~Van~Hemmen and W.~Wreszinski}, {\em Lyapunov function for the kuramoto
  model of nonlinearly coupled oscillators}, Journal of Statistical Physics, 72
  (1993), pp.~145--166.

\bibitem{varga1954eigenvalues}
{\sc R.~S. Varga}, {\em Eigenvalues of circulant matrices}, Pacific J. Math, 4
  (1954), pp.~151--160.

\bibitem{vathakkattil2020limits}
{\sc G.~Vathakkattil~Joseph and V.~Pakrashi}, {\em Limits on anti-phase
  synchronization in oscillator networks}, Scientific Reports, 10 (2020),
  pp.~1--9.

\bibitem{watanabe1997stability}
{\sc S.~Watanabe and J.~W. Swift}, {\em Stability of periodic solutions in
  series arrays of josephson junctions with internal capacitance}, Journal of
  nonlinear science, 7 (1997), pp.~503--536.

\bibitem{Wiley2006}
{\sc D.~A. Wiley, S.~H. Strogatz, and M.~Girvan}, {\em The size of the sync
  basin}, Chaos: An Interdisciplinary Journal of Nonlinear Science, 16 (2006),
  p.~015103.

\bibitem{Yoneda2021}
{\sc R.~Yoneda, T.~Tatsukawa, and J.~N. Teramae}, {\em {The lower bound of the
  network connectivity guaranteeing in-phase synchronization}}, Chaos, 31
  (2021).

\end{thebibliography}

\end{document}